\documentclass[twoside]{amsart}
\usepackage{latexsym}
\usepackage{enumitem}
\usepackage{amssymb,amsmath,amsopn}
\usepackage[dvips]{graphicx}   
\usepackage{color,epsfig}      
\usepackage{tikz} 
\usepackage{subcaption} 
\usepackage{array}
\usepackage{hyperref}

\renewcommand{\Re}{\mathbb R}

\newcommand{\BB}{\mathbf B}
\newcommand{\M}{\mathcal{M}}
\newcommand{\Sph}{\mathbb{S}}
\newcommand{\Sbf}{\mathbf{S}}

\newcommand{\F}{\mathcal{F}}
\newcommand{\E}{\mathcal{E}}

\newcommand{\T}{\mathcal{T}}

\DeclareMathOperator{\card}{card}
\DeclareMathOperator{\inter}{int}

\DeclareMathOperator{\conv}{conv}

\DeclareMathOperator{\perim}{perim}
\DeclareMathOperator{\area}{area}

\DeclareMathOperator{\cirr}{cr}
\DeclareMathOperator{\ir}{ir}

\DeclareMathOperator{\bd}{bd}

\theoremstyle{plain}
\newtheorem{theorem}{Theorem}[section]
\newtheorem{lemma}[theorem]{Lemma}
\newtheorem{conjecture}{Conjecture}
\newtheorem{corollary}[theorem]{Corollary}

\newtheorem{question}[theorem]{Question}
\newtheorem{remark}[theorem]{Remark}
\theoremstyle{definition}
\newtheorem{definition}[theorem]{Definition}
\numberwithin{equation}{section}
\newtheorem{example}[theorem]{Example}

\begin{document}

\title[The Honeycomb Conjecture]{The Honeycomb Conjecture in normed planes and an alpha-convex variant of a theorem of Dowker}

\author[Z. L\'angi]{Zsolt L\'angi}
\author[S. Wang]{Shanshan Wang}

\address{Zsolt L\'angi, Bolyai Institute, University of Szeged,\\
Aradi v\'ertan\'uk tere 1, H-6720, Szeged, Hungary and\\
HUN-REN Alfr\'ed R\'enyi Institute of Mathematics,\\
Re\'altanoda utca 13-15, H-1053, Budapest, Hungary}
\email{zlangi@server.math.u-szeged.hu}
\address{Shanshan Wang, Bolyai Institute, University of Szeged,\\
Aradi v\'ertan\'uk tere 1, H-6720, Szeged, Hungary}
\email{shanshanwang@server.math.u-szeged.hu}

\thanks{Partially supported by the ERC Advanced Grant ``ERMiD'', the National Research, Development and Innovation Office, NKFI, K-147544 grant, and project TKP2021-NVA-09 provided by the Ministry of Innovation and Technology of Hungary.}


\begin{abstract}
The Honeycomb Conjecture states that among tilings with unit area cells in the Euclidean plane, the average perimeter of a cell is minimal for a regular hexagonal tiling. This conjecture was proved by L. Fejes T\'oth for convex tilings, and by Hales for not necessarily convex tilings. In this paper we investigate the same question for tilings of a given normed plane, and show that among normal, convex tilings in a normed plane, the average squared perimeter of a cell is minimal for a tiling whose cells are translates of a centrally symmetric hexagon. We also show that the question whether the same statement is true for the average perimeter of a cell is closely related to an $\alpha$-convex variant of a theorem of Dowker on the area of polygons circumscribed about a convex disk. Exploring this connection we find families of norms in which the average perimeter of a cell of a tiling is minimal for a hexagonal tiling, and prove some additional related results. Finally, we apply our method to give a partial answer to a problem of Steinhaus about the isoperimetric ratios of cells of certain tilings in the Euclidean plane, appeared in an open problem book of Croft, Falconer and Guy.
\end{abstract}

\maketitle

\section{Introduction}\label{sec:intro}

The aim of this paper is to investigate certain properties of mosaics. A \emph{mosaic} or \emph{tiling} of the Euclidean $d$-space $\Re^d$ is a countable family of compact sets $\T$, called cells, with the property that $\bigcup \T = \Re^d$, and the interiors of any two distinct cells are disjoint. Here, for convenience, it is usually assumed that every cell has nonempty interior. A tiling is \emph{convex}, if every cell in it is convex; it is known that in this case every cell is a convex polytope \cite[Theorem 1]{Schulte}. If, for a convex tiling $\T$, there are universal constants $0 < \hat{r} < \hat{R}$ such that every cell of $\T$ contains a ball of radius $\hat{r}$ and is contained in a ball of radius $\hat{R}$, the tiling is called \emph{normal}. In this paper we deal only with normal, convex tilings.

One of the classic problems of geometry, regarding planar mosaics, is the Honeycomb Conjecture, which states, roughly speaking, that in a decomposition of the Euclidean plane into cells of equal area, the average perimeter of the cells is minimal for the regular hexagonal grid; i.e. a tiling in which every cell is a unit area regular hexagon (for an investigation of mathematically rigorous variants of this problem, see \cite{Morgan}). This conjecture first appeared in Roman time in a book of Varro about agriculture \cite{Varro}. Despite its ancient origin, it took a surprisingly long time to find a satisfactory proof for this conjecture: it was solved by L. Fejes T\'oth for normal, convex mosaics \cite{LFT} in the 1940s , while the most general version is due to Hales \cite{Hales1} in 2001, who dropped the condition of convexity; the proofs of both statements require heavy computation based on the metric properties of the Euclidean plane. Despite its solution, this problem still seems to be interesting for mathematicians \cite{BFVV, BF19}, engineers \cite{Zhang} and even for philosophers \cite{Raz}.

Clearly, it is a natural question to ask if a similar statement holds for tilings in any normed plane. The aim of this paper is to investigate this problem. To do it, we recall a few well-known facts about normed planes. In particular, we recall that every origin-symmetric convex disk $M$ (i.e. every compact, convex set $M$ with nonempty interior and satisfying $M=-M$) is the unit disk of a normed plane, and the unit disk of a normed plane is an origin-symmetric convex disk. In the paper, for a normed plane $\M$, we denote the unit disk of $\M$ by $M$, and the norm of a point $p \in \M$ by $||p||_M$. For convenience, we imagine that $\M$ also has an underlying Euclidean structure, and denote the Euclidean norm of $p$ by $||p||$.

We also note that any finite dimensional real normed space can be equipped with a Haar measure, and that this measure is unique up to multiplication of the standard Lebesgue measure by a scalar (cf. e.g. \cite{Thompson}). This scalar does not play a role in our investigation and in the paper $\area(\cdot)$ denotes $2$-dimensional Lebesgue measure. Finally, as usual, we define the perimeter of any convex disk $K$ in $\M$, also called the \emph{$M$-perimeter} of $K$ and denoted by $\perim_M(K)$, as the supremum of the sums of the $M$-lengths of the sides of any convex polygon inscribed in $K$ (see e.g. \cite{Schaffer}).

We ask the following question, where by a hexagonal tiling we mean a tiling whose cells are translates of a given centrally symmetric hexagon.

\begin{question}
Is it true that in any normal, convex tiling in any normed plane $\M$, with unit area cells, the average perimeter of a cell, if it exists, is minimal for a hexagonal tiling?
\end{question}

Before stating our first result, we note that there are optimization problems in discrete geometry whose solution for tilings in the Euclidean plane is a regular hexagonal tiling, while in any normed plane it is a (not necessarily regular) hexagonal tiling. As examples, we mention the \emph{densest circle packing problem} (see \cite{LFT2, FTL, Rogers, Rogers1} for its solution in the Euclidean plane and in any normed plane) as well as the \emph{simultaneous packing and covering problem} \cite{Zong}, which can be regarded as isoperimetric problems involving the area and inradius of cells, and the circumradius and inradius of cells of a tiling, respectively.

In the paper we need the following definition, where by $\BB^2$ we denote the Euclidean closed unit disk centered at the origin $o$, and $\card(\cdot)$ denotes the cardinality of a set.

\begin{definition}\label{defn:averageperim}
Let $\T$ be a convex, normal tiling in the normed plane $\M$. For any $R > 0$, let $\T(R)$ denote the family of cells of $\T$ contained in $R  \BB^2$. Let $\alpha > 0$. 
We define the \emph{lower average $\alpha$th powered perimeter} of a cell of $\T$ as the quantity
\begin{equation}\label{eq:averageperim}
\underline{P}_{\alpha}(\T) = \liminf_{R \to \infty} \frac{\sum_{C\in\T(R)} \left( \perim_M(C) \right)^{\alpha} }{\card(\T(R))}.
\end{equation}
We define the \emph{upper average $\alpha$th powered perimeter} of a cell of $\T$, denoted by $\overline{P}_{\alpha}(\T)$, in the same way, replacing $\liminf$ by $\limsup$.
If $\underline{P}_{\alpha}(\T) = \overline{P}_{\alpha}(\T)$, we call this quantity the \emph{average $\alpha$th powered perimeter} of a cell of $\T$, and denote it by $P_{\alpha}(\T)$.
We define the quantities $\underline{P}_{\log}(\T), \overline{P}_{\log}(\T)$ and $P_{\log}(\T)$ similarly, replacing $\left( \perim_M(C) \right)^{\alpha}$ by $\log \left( \perim_M(C) \right)$ in the above definitions.
\end{definition}

We note that, as usual in convex geometry, we may denote the quantities $\underline{P}_{\log}(\T)$, $\overline{P}_{\log}(\T)$ and $P_{\log}(\T)$ also by $\underline{P}_{0}(\T), \overline{P}_{0}(\T)$ and $P_{0}(\T)$, respectively. We also remark that whereas the above procedure is the standard way to define the `average' of a certain geometric quantity of infinitely many objects in the plane, the Euclidean disk $\BB^2$ in the definition of $\T(R)$ is often replaced in the literature by other figures, for instance squares or hexagons (see e.g. \cite{FTL, Lagerungen}).

\begin{remark}\label{rem:monotonicity}
For any $\alpha, \beta \in (0,\infty)$ with $\alpha < \beta$ and any normal, convex tiling $\T$ in $\M$, we have $\exp(\underline{P}_{\log}(\T)) \leq \left( \underline{P}_{\alpha}(\T) \right)^{1/\alpha} \leq \left( \underline{P}_{\beta}(\T) \right)^{1/\beta}$ and $\exp(\overline{P}_{\log}(\T)) \leq \left( \overline{P}_{\alpha}(\T) \right)^{1/\alpha} \leq \left( \overline{P}_{\beta}(\T) \right)^{1/\beta}$. Furthermore, if $\T$ is a hexagonal tiling, we have equality in all the previous inequalities.
\end{remark}

Our first result is the following theorem, which, as it is shown by Remark~\ref{rem:monotonicity}, can be regarded as a `weaker' variant of the Honeycomb Conjecture.

\begin{theorem}\label{thm:suboptimal}
For any normed plane $\M$ there is a hexagonal tiling $\T_{hex}$ of $\M$ with unit area tiles such that for any convex, normal tiling $\T$ of $\M$ with unit area tiles, we have
\[
\underline{P}_2(\T) \geq P_2(\T_{hex}).
\]
\end{theorem}

We will see that the proof of Theorem~\ref{thm:suboptimal} relies on the following classical result of Dowker \cite{CHD} on the area of convex polygons circumscribed about a convex disk.

\begin{theorem}[Dowker]\label{thm:Dowker1}
For any convex disk $K$ in $\Re^2$, let
\[
A_K(n)=\inf\{ \area(P):  P \hbox{ is a convex } n\hbox{-gon circumscribed about } K \}.
\]
Then the sequence $ \{ A_K(n)\}$ is convex. In other words, for any $n \geq 4$, we have
\[
A_K(n-1)+A_K(n+1)\geq 2 A_K(n).
\]
\end{theorem}

Based on this result, we define two properties.

\begin{definition}\label{defn:alphahoneycomb}
Let $\alpha \in (0,\infty)$. We say that the normed plane $\M$ satisfies the \emph{$\alpha$-honeycomb property}, if there is a hexagonal tiling $\T_{hex}$ of $\M$ such that for any convex, normal tiling $\T$ of $\M$, we have
\[
\underline{P}_{\alpha}(\T) \geq P_{\alpha}(\T_{hex}).
\]
Similarly, we say that it satisfies the \emph{log-honeycomb} (or \emph{$0$-honeycomb}) \emph{property} if the same holds for the lower average log-perimeter of a cell of $\T$.
\end{definition}

Observe that by Theorem~\ref{thm:suboptimal}, any norm with unit disk $M$ satisfies the $2$-honeycomb property.

\begin{definition}\label{defn:alphadowker}
Let $\alpha \in (0,\infty)$. We say that a convex disk $K$ satisfies the \emph{$\alpha$-Dowker property} if the sequence $\{ A_K^{\alpha}(n) \}$ is convex. Furthermore, we say that it satisfies the \emph{log-Dowker} (or \emph{$0$-Dowker}) \emph{property} if the sequence $\{ \log A_K(n) \}$ is convex.
\end{definition}

In our proof we only need the following, weaker variant of the $\alpha$-Dowker property.

\begin{definition}\label{defn:weakalphadowker}
Let $\alpha \in (0,\infty)$. We say that a convex disk $K$ satisfies the \emph{weak $\alpha$-Dowker property} if
\begin{equation}\label{eq:weakalphadowker}
\frac{n-6}{n-m} A_K^{\alpha}(m) + \frac{6-m}{n-m} A_K^{\alpha}(n) \geq A_K^{\alpha}(6)
\end{equation}
holds for any $3 \leq m < 6 < n$. Similarly, we say that $K$ satisfies the \emph{weak log-Dowker} (or \emph{weak $0$-Dowker}) \emph{property} if
\begin{equation}\label{eq:weakalphadowker2}
\frac{n-6}{n-m} \log A_K(m) + \frac{6-m}{n-m} \log A_K(n) \geq \log A_K(6)
\end{equation}
holds for any $3 \leq m < 6 < n$.
\end{definition}

For brevity, if $\alpha=1$ in any of the above definitions, we may omit it from the notation.
In our next result, we recall the notion of \emph{isoperimetrix} of a normed plane $\M$ with unit disk $M$, defined as the rotated copy of the Euclidean polar of $M$, about the origin $o$, by $\frac{\pi}{2}$. We denote the isoperimetrix of $\M$ by $M_{iso}$.

\begin{theorem}\label{thm:generalized}
Let $M$ be an $o$-symmetric convex disk. For any $\alpha \in [ 0,\infty)$, if the isoperimetrix $M_{iso}$ of $\M$ satisfies the weak $\alpha$-Dowker property, then the normed plane $\M$ satisfies the $(2\alpha)$-honeycomb property.
\end{theorem}

It is an elementary exercise to check that $A_{\BB^2}(n)=n \tan \frac{\pi}{n}$, implying that $\BB^2$ satisfies the log-Dowker property. Thus, Theorem~\ref{thm:generalized} immediately yields Corollary~\ref{cor:Euclidean}.

\begin{corollary}\label{cor:Euclidean}
The Euclidean plane satisfies the log-honeycomb property.
\end{corollary}

The structure of the paper is as follows.

In Section~\ref{sec:main}, we prove Theorems~\ref{thm:suboptimal} and \ref{thm:generalized}. In Sections~\ref{sec:polygonal} and \ref{sec:general} we investigate the $\alpha$-Dowker and $\alpha$-honeycomb properties of convex disks and normed planes, respectively. More specifically, in Section~\ref{sec:polygonal} we consider convex polygons and polygonal norms, while in Section~\ref{sec:general} we examine properties of convex disks and norms which are not necessarily polygonal. In Section~\ref{sec:Steinhaus} we consider a problem of Steinhaus, appeared in the problem book \cite{CFG} of Croft, Falconer and Guy, asking if it is true that among tilings of the Euclidean plane with tiles whose diameters are at least some universal constant $D > 0$, the maximum isoperimetric ratio of the cells is minimal for a regular hexagonal tiling. Using the approach of the proof of Theorem~\ref{thm:suboptimal}, we show that for convex, normal tilings, a stronger statement holds in any normed plane, and also investigate possible generalizations of our method for this problem. We finish the paper with an additional remark in Section~\ref{sec:remark}.

In our investigation, we denote by $\inter(X)$, $\bd(X)$ and $\conv(X)$ the \emph{interior, boundary} and the \emph{convex hull} of a set $X$, and for any two points $p,q$, we denote by $[p,q]$ the closed straight line segment connecting them. Finally, we denote the Euclidean unit circle $\bd \BB^2$ by $\Sph^1$.

\section{Proof of Theorems~\ref{thm:suboptimal} and \ref{thm:generalized}}\label{sec:main}

First, we observe that Theorem~\ref{thm:generalized}, combined with Theorem~\ref{thm:Dowker1} readily implies Theorem~\ref{thm:suboptimal}. Thus, we only need to prove Theorem~\ref{thm:generalized}. Throughout this section, we deal with a given convex, normal tiling $\T$ in a normed plane $\M$. We permit the angles of cells to be straight angles; due to this, we can regard $\T$ as an \emph{edge-to-edge} tiling, in which every edge of every cell belongs to exactly one more cell. During our investigation, $\T(R)$ denotes the family of cells of $\T$ contained in $R \BB^2$, and for every $C \in \T$, $v(C)$ denotes the number of sides of $C$.
We start with collecting some elementary observations and known facts about tilings and normed planes.

\begin{lemma}\label{lem:maxsides}
Let $0 < \hat{r} < \hat{R}$ be given such that every cell of $\T$ contains a Euclidean disk of radius $\hat{r}$ and is contained in a Euclidean disk of radius $\hat{R}$. Then every cell $C \in \T$ has at most $\frac{9\hat{R}^2}{\hat{r}^2}-1$ neighbors.
\end{lemma}

\begin{proof}
For every cell $C \in \T$, there is a Euclidean disk of radius $3\hat{R}$ that contains $C$ as well as all its neighbors. The area of every cell is at least $\hat{r}^2 \pi$. Thus, the assertion follows from a simple area estimate.
\end{proof}

\begin{lemma}\label{lem:discrepancy}
Let $\T_{bd}(R)$ denote the family of cells of $\F$ that intersect the circle $R \Sph^1$. Then $\card (\T_{bd}(R)) = \Theta(R)$.
\end{lemma}

\begin{proof}
Since the diameter of every cell $C \in \T$ is at most $2\hat{R}$, every cell is contained in the closure of the set $(R+2\hat{R}) \BB^2 \setminus (R-2\hat{R}) \BB^2$. Thus,
$\card (\T_{bd}(R)) \leq \frac{(R+2\hat{R})^2 \pi - (R-2\hat{R})^2 \pi}{\hat{r}^2 \pi} = \frac{8 \hat{R} R}{\hat{r}^2}$, implying that $\card (\T_{bd}(R)) = \mathcal{O}(R)$. On the other hand, the union of these cells covers $R \Sph^1$, and the diameter of each cell is at most $2\hat{R}$. Thus, $\card (\T_{bd}(R)) \geq \frac{2 \pi}{2 \arcsin \frac{\hat{R}}{R}}$, implying that $\card (\T_{bd}(R)) = \Omega(R)$.
\end{proof}

\begin{definition}\label{defn:averagev}
We define the \emph{lower average number of sides} of a cell of $\T$ as
\[
\underline{v}(\T) = \liminf_{R \to \infty} \frac{\sum_{C\in\T(R)} v(C) }{\card(\T(R))}.
\]
We define the \emph{upper average number of sides} $\overline{v}(\T)$ of a cell of $\T$ in the same way, replacing $\liminf$ by $\limsup$. If $\underline{v}(\T)=\overline{v}(\T)$, we call this quantity the \emph{average number of sides} of a cell of $\T$.
\end{definition}

\begin{lemma}\label{lem:Euler}
For any $R > 2\hat{R}$, set $v_R(\T) = \frac{\sum_{C\in\T(R)} v(C) }{\card(\T(R))}$. Then, $v_R(\T) \leq 6 + \mathcal{O}\left( \frac{1}{R} \right)$, and $\overline{v}(\T) \leq 6$.
\end{lemma}

\begin{proof}
Consider the graph $G$ whose vertices are the vertices of the cells of $\T(R)$, and two vertices are connected by an edge if they are the endpoints of an edge of a cell in $\T(R)$. Then $G$ is a planar graph. Let $V(G)$, $E(G)$ and $F(G)$ denote the numbers of vertices, edges and faces of $G$. Furthermore, let $\E_{bd}(G)$ denote the family of the edges of $G$ in the boundary of $\bigcup \T(R)$. We note that every edge in $\E_{bd}(G)$ belongs to a cell of $\T$ intersecting $R \Sph^1$, and thus, it follows from Lemmas~\ref{lem:maxsides} and \ref{lem:discrepancy} that $E_{bd}(G) := \card (\E_{bd}(G))=\mathcal{O}(R)$. These lemmas, an area estimate and the observation that $\bigcup \T(R)$ contains $(R-2\hat{R}) \BB^2$ and is contained in $(R+2\hat{R}) \BB^2$ shows that $V(G)$, $E(G)$ and $F(G)$ are all of order of magnitude $\Theta(R^2)$.

We observe that there might be vertices of $G$ whose degree is less than $3$. If $p$ is such a vertex, then its degree is $2$, it lies in $\bd \left( \bigcup \T(R) \right)$, and it belongs to two edges in $\E_{bd}(G)$. Thus, the number of such edges is at most $E_{bd}(G) = \mathcal{O}(R)$. We merge these edges into one edge (i.e. we apply SP-reduction) and continue this process until we obtain another planar graph $G'$ in which the degree of every vertex is at least $3$. If $V(G'), E(G')$ and $F(G')$ denote the numbers of vertices, edges and faces of $G'$, then by the construction of $G'$ we have $V(G') \leq V(G)$ and $E(G') \leq E(G)$ with $V(G)-V(G')=E(G)-E(G') = \mathcal{O}(R)$, and $F(G')=F(G)$.

Applying Euler's formula for $G'$ and using the fact that the degree of every vertex of $G'$ is at least $3$ yields that $E(G') \leq 3 F(G') - 6$, implying that $E(G) \leq 3 F(G) + \mathcal{O}(R)$. On the other hand, $v_R(\T) \leq \frac{2 E(G)}{F(G)}$, and as $E(G)$ and $F(G)$ are of $\Theta(R^2)$, it follows that $v_R(\T) \leq 6 + \mathcal{O}\left( \frac{1}{R} \right)$, which readily yields the assertion.
\end{proof}

The original isoperimetric inequality, observed by Zenodorus in ancient Greece, states that among regions of a given perimeter in the Euclidean plane, Euclidean disks have maximum area. This result was generalized for any normed plane by Busemann \cite{Busemann} in the following form. 

\begin{theorem}[Busemann]\label{thm:Busemann}
Let $\M$ be a normed plane. The area enclosed by a simple, closed curve $\Gamma$ of a given $M$-length is maximized if $\Gamma$ is the boundary of a plane convex body $K$ homothetic to the so-called \emph{isoperimetrix} $M_{iso}$ of $\M$, obtained as the polar of the rotated copy of the unit disk $M$ of $\M$ by $\frac{\pi}{2}$ (see Figure~\ref{fig:isoperimetrix}).
\end{theorem}

\begin{figure}[ht]
\begin{center}
\includegraphics[width=0.8\textwidth]{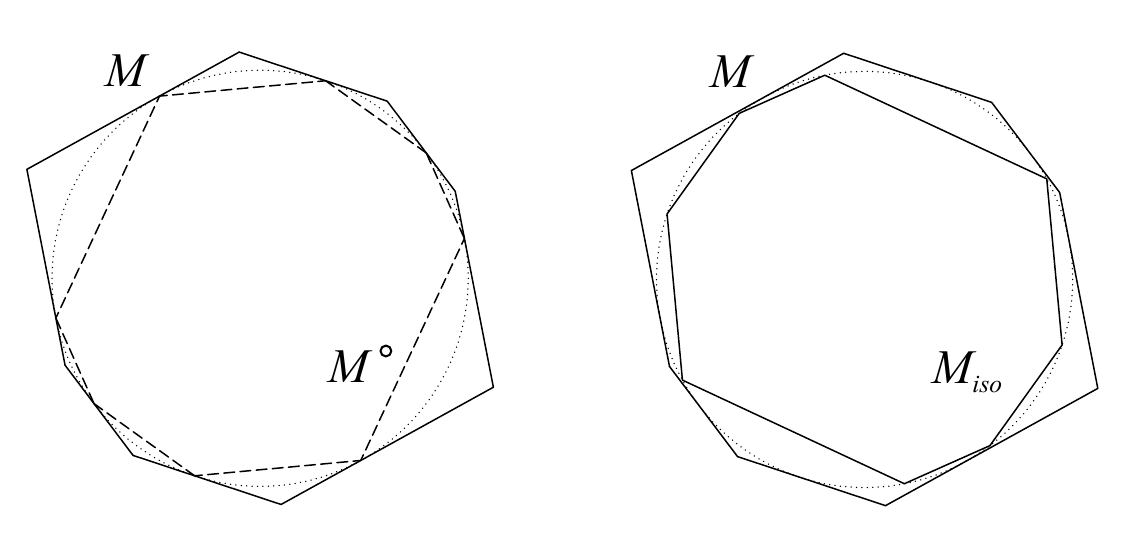}
\caption{The isoperimetrix $M_{iso}$ of a norm with unit disk $M$. The dotted circle is the Euclidean unit disk $\BB^2$ centered at $o$. The left-hand side panel shows $M$ and its polar $M^{\circ}$, the isoperimetrix in the right-hand side panel is a rotated copy of $M^{\circ}$ by $\frac{\pi}{2}$}.
\label{fig:isoperimetrix}
\end{center}
\end{figure}

We note that by Theorem~\ref{thm:Busemann}, the isoperimetrix of the Euclidean plane is a Euclidean circular disk.

A variant of this inequality, in Euclidean plane, states that among convex polygons with given perimeter and fixed directions of their sides, the ones circumscribed about a circular disk have maximal area. This statement is called L'Huilier's inequality, which is generalized for arbitrary normed planes by Chakerian \cite{GDC} in the following way.

\begin{theorem}[Chakerian]\label{thm:Chakerian}
Let $\M$ be a normed plane with unit disk $M$. Let the isoperimetrix of the plane be $M_{iso}$. Let $K$ be an arbitrary convex $n$-gon in $\M$, and let $K^*$ be the convex $n$-gon circumscribed about $M_{iso}$ whose sides have the same outer unit normals as the sides of $K$. Let the $M$-perimeter of $K$ be $L$, the area of $K$ be $F$, and the area of $K^*$ be $f$. Then
\[
L^2 - 4fF \geq 0,
\]
with equality if and only if $K$ is homothetic to $K^*$. 
\end{theorem}

Now we are ready to prove Theorem~\ref{thm:generalized}. We assume that $M_{iso}$ satisfies the weak $\alpha$-Dowker property for some $\alpha \in (0,\infty)$; for the case $\alpha =0$ a straightforward modification of our argument can be applied.
From now on we also assume that every cell of $\T$ has unit area. We need to consider the following quantity.
\[
\underline{P}_{2\alpha}(\T) = \liminf_{R \to \infty} \frac{\sum_{C\in\T(R)} \left( \perim_M(C) \right)^{2\alpha} }{\card(\T(R))}.
\]
For any $C \in \T(R)$, let $C^*$ denote the convex polygon circumscribed about $M_{iso}$ such that the sides of $C$ and $C^*$ have the same outer unit normals, and let $v(C)$ denote the number of sides of $C$.
Then, by Theorem~\ref{thm:Chakerian}, and since $\area(C)=1$ for any $C \in \T(R)$, we have
\[
\underline{P}_{2\alpha}(\T) \geq 4^{\alpha}\liminf_{R \to \infty} \frac{\sum_{C \in \T(R)} \left( \area(C^*) \right)^{\alpha} }{\card \T(R)} \geq 4^{\alpha} \liminf_{R \to \infty} \frac{\sum_{C \in \T(R)} \left( A_{M_{iso}}(v(C)) \right)^{\alpha} }{\card \T(R)}.
\]
where $A_{M_{iso}}(v(C))$ denotes the area of a minimum area polygon with $v(C)$ sides, circumscribed about $M_{iso}$ (see Theorem~\ref{thm:Dowker1}).
By Lemma~\ref{lem:Euler}, $\overline{v}(\T) \leq 6$. Now, we leave to the reader that the fact that $A_{M_{iso}}(v(C))$ is weak $\alpha$-Dowker and monotone decreasing yields the following: If $v_1, \ldots, v_k \geq 3$ are real numbers and $0 \leq \tau_1, \tau_2 \ldots, \tau_k \leq 1$ are weights satisfying $\sum_{i=1}^k \tau_i=1$ and $\sum_{i=1}^k \tau_k v_k \leq 6$, then $\sum_{i=1}^k \tau_i \left( A_{M_{iso}}(v_i) \right)^{\alpha} \geq \left( A_{M_{iso}}(6) \right)^{\alpha}$. 
Based on this, it follows that
\[
\underline{P}_{2\alpha}(\T) \geq 4^{\alpha} \left( A_{M_{iso}}(6) \right)^{\alpha}.
\]
Now the only thing we need to show is that for some hexagonal tiling $\T_{hex}$ of $\M$, we have $P_{2\alpha}(\T_{hex}) = 4^{\alpha} \left( A_{M_{iso}}(6) \right)^{\alpha}$.
To do it we recall another theorem of Dowker from \cite{CHD}.

\begin{theorem}[Dowker]\label{thm:Dowker2}
Let $K$ be an $o$-symmetric convex disk. Then, for every $m \geq 2$, there is an $o$-symmetric convex $(2m)$-gon $P$ circumscribed about $K$ with $\area(P)=A_K(2m)$.
\end{theorem}

By Theorem~\ref{thm:Dowker2}, there is an $o$-symmetric hexagon $H$ circumscribed about $M_{iso}$ of area $A_{M_{iso}}(6)$. Let $\T_{hex}$ be a tiling whose cells are translates of a homothetic copy of $H$ of unit area. Then the desired equality follows from the equality part of Theorem~\ref{thm:Chakerian}.

\begin{remark}\label{rem:Morgan}
If the isoperimetrix $M_{iso}$ of the norm has at most six sides (which is equivalent to the property that $M$ has at most six sides), then $M$ satisfies the log-honeycomb property.
\end{remark}

\section{Results about polygonal norms}\label{sec:polygonal}

The goal of this section is to investigate the $\alpha$-honeycomb properties of polygonal norms, and the $\alpha$-Dowker properties of polygons.

\subsection{An algorithm to check the $\alpha$-honeycomb property of a polygonal norm}\label{subsec:algorithm}

There is no straightforward way to check whether in a specific normed plane the honeycomb conjecture is true or not. Nevertheless, the sufficient condition in Theorem~\ref{thm:generalized} can be checked using a suitable algorithm.
In this subsection we give a simple algorithm that checks if the isoperimetrix of a polygonal norm satisfies the $\alpha$-Dowker, or weak $\alpha$-Dowker property for some given $\alpha \in [0,\infty)$. We assume that $\alpha > 0$, and note that our argument can be modified in a natural way for the log-Dowker or weak log-Dowker property. 

\begin{remark}\label{rem:checkingDowker}
If the unit disk $M$ of a norm is a $(2k)$-gon, then $M_{iso}$ is also a $(2k)$-gon, and for any $n \geq 2k$, we clearly have $A_M(n)=\area(M_{iso})$. Thus, to check if $M^*$ satisfies the $\alpha$-Dowker property, it is sufficient to check if the inequality
\begin{equation}\label{eq:algorithm}
A_{M_{iso}}^{\alpha}(n-1) + A_{M_{iso}}^{\alpha}(n+1) \geq 2 A_{M_{iso}}^{\alpha}(n)
\end{equation}
holds for all values $4 \leq n \leq 2k$.
Similarly, to check if $M_{iso}$ satisfies the weak $\alpha$-Dowker property, it is enough to check if the inequality in (\ref{eq:weakalphadowker}) is satisfied for $m=3,4,5$ and $6 < n \leq 2k$.
\end{remark}

We also recall that there is an algorithm, given by Aggarwal, Chang and Yap \cite{ACY85} that computes, for any given convex polygon $P$ with $s$ sides, the area of a minimum area circumscribed convex $t$-gon for any given $3 \leq t < s$, in $\mathcal{O}(s^2 \log s \log t)$ steps.

Thus, given an $o$-symmetric $(2k)$-gon $M$, we can provide the following algorithm to check if, say, the weak $\alpha$-Dowker property is satisfied for $M_{iso}$.

\begin{enumerate}
\item[Step 1]: We compute $M_{iso}$. Since $M_{iso}$ is the rotated copy of the polar of $M$ by $\frac{\pi}{2}$, and since the polar of an $o$-symmetric convex $(2k)$-gon can be computed in $\mathcal{O}(k)$ steps, this step requires $\mathcal{O}(k)$ steps.
\item [Step 2]: Using the algorithm in \cite{ACY85}, we compute the quantities $A_{M_{iso}}(n)$ for all $3 \leq n \leq 2k$. This step requires $\mathcal{O}(k^3 \log^2 k)$ steps.
\item[Step 3]: We check the inequality in (\ref{eq:weakalphadowker}) for $m=3,4,5$ and all $6 < n \leq 2k$, in $\mathcal{O}(k)$ steps.
\end{enumerate}

To check the $\alpha$-Dowker property of $M_{iso}$, in Step 3 we check the inequality in (\ref{eq:algorithm}) in $\mathcal{O}(k)$ steps.

\subsection{Norms where the unit disk is a regular polygon}\label{subsec:regpol}

In this subsection we investigate which normed planes, with an $o$-symmetric regular polygon as their unit disk, satisfy the honeycomb property.
Our main result is as follows.

\begin{theorem}\label{thm:regularmain}
If the unit disk of a normed plane $\M$ is a regular $(2k)$-gon with $k \neq 4,5,7$, then $\M$ satisfies the honeycomb property.
\end{theorem}

Note that if the unit disk $M$ of $\M$ is a regular $(2k)$-gon, then $M_{iso}$ is also a regular $(2k)$-gon. Thus,
Theorem~\ref{thm:regularmain} readily follows from combining Theorems~\ref{thm:regularDowker} and \ref{thm:generalized}.

\begin{theorem}\label{thm:regularDowker}
A regular $(2k)$-gon $P_k$, with $k \geq 2$, satisfies the weak $\frac{1}{2}$-Dowker property if and only if $k \neq 4,5,7$.
\end{theorem}

Consider a normed plane $\M$ whose unit disk is a regular $(2k)$-gon with $k=4,5$ or $7$. A possible interpretation of Theorem~\ref{thm:regularDowker} is that in this case to minimize the `average' perimeter it is not worth using `optimal' hexagons. Nevertheless, it seems unlikely that one can fit `optimal' $k$-gons with more than one different values of $k$ to tile the plane. This observation might indicate that these planes may also satisfy the honeycomb property.

\begin{remark}\label{rem:onlyweak}
According to Lemma~\ref{lem:regmain}, if $k \geq 4$, then $A_{2k-2}(P_k)$ is a convex combination of $A_{2k-1}(P_k)$ and $A_{2k-3}(P_k)$, implying that in this case $P_k$ does not satisfy the $\alpha$-Dowker property for any $\alpha < 1$.
\end{remark}

To prove Theorem~\ref{thm:regularDowker}, we first find the values of $A_{n}(P_k)$ of the regular $(2k)$-gon $P_k$, with $k \geq 2$ and $3 \leq n \leq 2k$. We note also that for any $n \geq 2k$, we have $A_n(P_k) = \area(P_k)$. 
In the following we assume that the regular polygon $P_k$ is circumscribed about the unit disk $\BB^2$. We denote the quantity $A_n(P_k)$ by $A(k,n)$.
For our investigation, we need Lemma~\ref{lem:minimumarea}, which is a slightly stronger version of Lemma 1 from \cite{ACY85}.

\begin{lemma}\label{lem:minimumarea}
Let $3 \leq n \leq m$. Let $P$ be a convex $m$-gon, and let $Q$ denote a minimum area convex $n$-gon circumscribed about $P$. Then the midpoint of every side of $Q$ belongs to $P$. Furthermore, if $n \geq 4$, then there is a minimum area convex $n$-gon $Q'$ such that there is at most one side of $Q'$ that does not contain a side of $P$, and if $S$ is such a side of $Q'$, then the sum of the angles of $Q'$ on $S$ is strictly less than $\pi$.
\end{lemma}

\begin{proof}
Clearly, every side of $Q$ intersects $P$.
Suppose that there is a side $S$ of $Q$ such that the midpoint of $S$ does not belong to $P$. Let $q$ be the point of $P \cap S$ closest to the midpoint of $S$. Then slightly rotating the sideline of $S$ around $q$ in a suitable direction yields a circumscribed polygon whose area is strictly smaller than $A_n(P)$, a contradiction (see Figure~\ref{fig:circumscribed}).

\begin{figure}[ht]
\begin{center}
\includegraphics[width=0.75\textwidth]{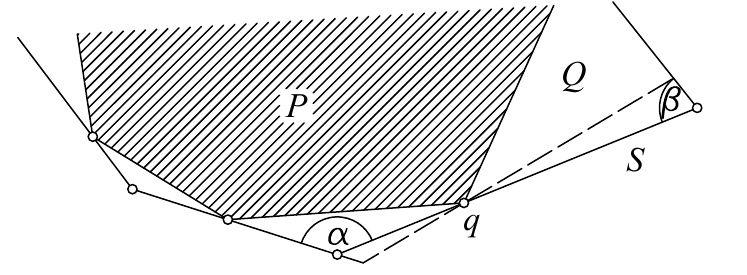}
\caption{Slightly rotating the side $S$ of $Q$ around $q$. The rotated copy is shown with a dashed line. Note that if the longer part of $S$ moves towards $P$, then the area of $Q$ decreases.}
\label{fig:circumscribed}
\end{center}
\end{figure}

Now, let $Q'$ be a minimum area convex $n$-gon circumscribed about $P$ with the property that it has the smallest number of sides that contain only a vertex of $P$. Let $S$ be such a side of $Q'$. Then, by the previous paragraph, the vertex $p$ of $P$ on $S$ is the midpoint of $S$. Let $\alpha$ and $\beta$ denote the two angles of $Q'$ on $S$. If $\alpha + \beta > \pi$, then slightly rotating the sideline of $Q'$ through $S$ about $p$ decreases the area of $Q'$, contradicting the assumption that $Q'$ is a minimum area circumscribed $n$-gon. If $\alpha + \beta = \pi$, then the same rotation does not change the area of $Q'$, and thus, rotating this line until it first reaches another vertex of $P$ we obtain a minimum area circumscribed polygon having strictly less sides than $Q'$ that contain only a vertex of $P$. Thus, we have $\alpha + \beta < \pi$. But any convex $n$-gon with $n \geq 4$ has at most one side satisfying this property.
\end{proof}

\begin{remark}\label{rem:osymmetricP}
We note that if $P$ is $o$-symmetric and $n$ is even, then the method of the proof of Lemma~\ref{lem:minimumarea}, combined with Theorem~\ref{thm:Dowker2}, yields that there is a minimum area convex $n$-gon $Q$ circumscribed about $P$ such that $Q$ is $o$-symmetric, and every side of $Q$ contains a side of $P$.
\end{remark}

The main lemma to prove Theorem~\ref{thm:regularDowker} is Lemma~\ref{lem:regmain}.

\begin{lemma}\label{lem:regmain}
Let $k \geq 2$ and $3 \leq n \leq 2k$ be positive integers, and let $2k=nq+r$, where $q, r$ are nonnegative integers, and $r < n$. If $n \geq 4$, then
\begin{equation}\label{eq:reg1}
A(k,n)= (n-r) \tan \frac{\pi q}{2k} + r \tan \frac{\pi (q+1)}{2k}.
\end{equation}
Furthermore,
\begin{equation}\label{eq:reg2}
A(k,3)=
\begin{cases}
        3 \tan \frac{\pi}{3}, & \text{if}~2k=3q,\\
        \frac{2}{\cos^2\left( \frac{\pi}{2k} \right)}\left( 2\sin\left( \frac{q\pi}{k} \right)+\sin\left( \frac{(q+1)\pi}{k} \right) \right), & \text{if}~2k=3q+1,\\
        \frac{2}{\cos^2\left( \frac{\pi}{2k} \right)}\left( \sin\left( \frac{q\pi}{k} \right)+2 \sin\left( \frac{(q+1)\pi}{k} \right) \right), & \text{if}~2k=3q+2.
    \end{cases}
\end{equation}
\end{lemma}

\begin{proof}
Let $Q$ be a minimum area $n$-gon circumscribed about $P_k$ with the smallest number of sides that contain only a vertex of $P_k$. By Lemma~\ref{lem:minimumarea} and Remark~\ref{rem:osymmetricP}, we can assume that if $n$ is even, then every side of $Q$ contains a side of $P_k$, and otherwise at most one side of $Q$ contains only a vertex of $P_k$.
Let the sides of $P_k$ be denoted as $S_1, S_2, \ldots, S_{2k}$ in counterclockwise order in $\bd(P_k)$.

We distinguish two cases.

\noindent
\emph{Case 1}, $n \geq 4$.

Set $B(k,n)= (n-r) \tan \frac{\pi q}{2k} + r \tan \frac{\pi (q+1)}{2k}$. We intend to show that $A(k,n)=B(k,n)$ for every choice of $k,n$.

\noindent
\emph{Subcase 1.1}: every side of $Q$ contains a side of $P_k$.
Without loss of generality, we may assume that the indices of the sides of $P_k$ contained in a side of $Q$ are of the form $s_1, s_1+s_2, \ldots, s_1+s_2+\ldots + s_n=2k$ for some positive integers $s_i$, where $s_i < k$ for every value of $i$. Since $P_k$ is circumscribed about $\BB^2$, we have
\[
\area(Q) = \sum_{i=1}^n \tan \frac{s_i \pi}{2k}.
\]
Note that as the function $x \mapsto \tan x$ is strictly convex on the domain $\left( 0, \frac{\pi}{2} \right)$, it follows that $\area(Q) \geq B(k,n)$, and equality is attained with a suitable choice of the $s_i$.

\noindent
\emph{Subcase 1.2}, there is a unique side of $Q$ that contains only a vertex of $P_k$. Observe that by our assumptions, in this case $n\geq 5$ is odd. By Lemma \ref{lem:minimumarea}, it follows that the sum of the angles $\alpha, \beta$ of $Q$ on the above side of $Q$ satisfies $\alpha + \beta < \pi$, and thus, at least one half of the sides of $P_k$ do not lie in $\bd(Q)$, implying that $k \geq n \geq 5$.
Let the sides of $Q$ be denoted by $E_1, E_2, \ldots, E_n$ in counterclockwise order. For $i=2,\ldots, n$, let $2\alpha_i$ denote the turning angle of $Q$ at the common vertex of $E_{i-1}$ and $E_i$, and let $2\alpha_1$ denote the turning angle at the common vertex of $E_n$ and $E_1$. Without loss of generality, let $E_1$ be the side of $Q$ that contains only a vertex of $P_k$. Then, for any $i$ with $3 \leq i \leq n$, $\alpha_i$ is an integer multiple of $\frac{\pi}{2k}$, and $\sum_{i=1}^n \alpha_i = \pi$.
Since $P_k$ is circumscribed about $\BB^2$, the distance of the line of $E_1$ from $o$ is strictly greater than $1$.
Let us translate the line of $E_1$ towards $o$ until its distance is $1$, and let $Q'$ denote the convex $n$-gon obtained in this way (see Figure~\ref{fig:regpol}). Then
\[
\area(Q) > \area(Q') = \sum_{i=_1}^n \tan \alpha_i.
\]
By Lemma~\ref{lem:minimumarea}, we have that $\alpha_1 + \alpha_2 > \frac{\pi}{2}$. Since this quantity is an integer multiple of $\frac{\pi}{2k}$, it follows that $\alpha_1 + \alpha_2 \geq \frac{\pi}{2}+\frac{s\pi}{2k}$ for some positive integer $s$. Without loss of generality, we assume that $\alpha_1 \leq \alpha_2$.

\begin{figure}[ht]
\begin{center}
\includegraphics[width=0.75\textwidth]{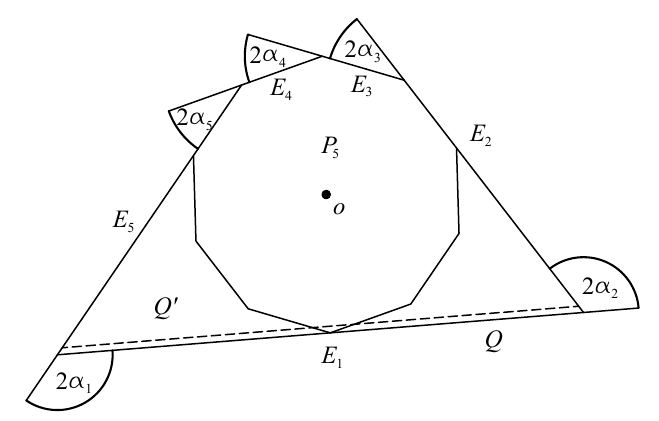}
\caption{An illustration for Case 2 of the proof of Lemma~\ref{lem:regmain} with $n=k=5$.}
\label{fig:regpol}
\end{center}
\end{figure}

We intend to construct a convex polygon circumscribed about $P_k$ with area smaller than $\area(Q')$. To do it, we assume that $s=1$; for the case $s\geq 2$ a straightforward modification of our argument can be given. First, note that as $E_1$ contains no side of $P_k$, we have that $\alpha_2 - \alpha_1  \neq \frac{\pi}{2k}$.
Consider the case that $\alpha_2-\alpha_1 > \frac{\pi}{2k}$, then there are integer multiples $\alpha_2' \geq \alpha_1'$ of $\frac{\pi}{2k}$ such that $\alpha_1'+\alpha_2'=\alpha_1+\alpha_2$. Then, by the convexity of the $x \mapsto \tan x$ function the expression $\tan \alpha_1' + \tan \alpha_2' + \sum_{i=_3}^n \tan \alpha_i$ is the area of a convex $n$-gon $Q''$ circumscribed about $P_k$, and $\area(Q'') < \area(Q') < \area(Q)$, contradicting our assumption about the minimality of the area of $Q$.

Assume that $\alpha_2-\alpha_1 < \frac{\pi}{2k}$. Then $\frac{\pi}{4} < \alpha_1 \leq \frac{\pi}{4}+\frac{\pi}{4k} \leq  \alpha_2 < \frac{\pi}{4}+\frac{\pi}{2k}$. Since $\sum_{i=3}^n \alpha_i = \frac{\pi}{2} - \frac{\pi}{2k}$ and $5 \leq n \leq k$, for at least one value of $i$, with $3 \leq i \leq n$, we have
\[
\alpha_i \leq \frac{1}{n-2} \left( \frac{\pi}{2} - \frac{\pi}{2k} \right) \leq \frac{\pi}{6}- \frac{\pi}{6k}.
\]
Without loss of generality, let $\alpha_3$ satisfy this inequality. Then, as $k \geq 5$, $\alpha_1-\alpha_3 > \frac{\pi}{4} - \left( \frac{\pi}{6}- \frac{\pi}{6k} \right) > \frac{\pi}{2k}$. Thus, setting $\alpha_1'= \alpha_1 - \frac{\pi}{2k}$, $\alpha_3' = \alpha_3 + \frac{\pi}{2k}$, by the convexity of the $x \mapsto \tan x$ function, we have that $\sum_{i=1}^n \tan \alpha_i > \tan \alpha_1' + \tan \alpha_2+ \tan \alpha_3' + \sum_{i=4}^n \tan \alpha_i$. Then, applying the argument in the previous case we obtain a convex $n$-gon $Q''$ circumscribed about $P_k$ such that $\area(Q'') < \area(Q')$, a contradiction.

This shows that $A(k,n)=B(k,n)$ for any $k \geq 2$ and $4 \leq n \leq 2k$.

\noindent
\emph{Case 2}, $n = 3$.

Let $Q$ be a minimum area triangle circumscribed about $P_k$. 
By Lemma~\ref{lem:minimumarea}, the midpoint of every side of $Q$ belongs to $P$. We denote the vertices of $P_k$ by $p_1, p_2, \ldots, p_{2k}$ in counterclockwise direction.
Depending on the number of sides of $Q$ that contain only a vertex of $P_k$, we distinguish four subcases.

\noindent
\emph{Subcase 2.1}, every side of $Q$ contains only a vertex of $P_k$. Let these vertices of $P_k$ be $p_i, p_j, p_l$, where we can assume that $1 \leq i < j < l \leq 2k$. Since these vertices are the midpoints of the corresponding sides of $Q$, $\conv \{ p_i, p_j, p_l \}$ is the midtriangle of $Q$ and hence, the sides of this triangle are parallel to the sides of $Q$. Since $Q$ is circumscribed about $P_k$, its sidelines support $P_k$ at the corresponding vertices of $P_k$. Characterizing which diagonals of $P_k$ are parallel to some supporting lines at $p_i$, $p_j$ and $p_k$, respectively, we obtain that $j-i, l-j$ and $i+2k-l$ differ by at most one. Thus, by an elementary computation we have that if $Q$ is a minimal area triangle circumscribed about $P_k$, and $\bd (Q)$ contains no side of $T$, then, with the notation $2k=3q+r$, where $q,r$ are nonnegative integers and $r \leq 2$,
\begin{equation}\label{eq:guess}
\area(Q)=\frac{2}{\cos^2 \left( \frac{\pi}{2k}\right) } \left( (3-r)\sin \left( \frac{q\pi}{k} \right) + r \sin \left( \frac{(q+1)\pi}{k} \right) \right). 
\end{equation}

\noindent
\emph{Subcase 2.2}, exactly one side of $Q$ contains a side of $P_k$. Let the sides of $Q$ be $E_1, E_2, E_3$, and assume that $[p_{2k},p_1] \subset E_1$, and $p_i \in E_2$ and $p_j \in E_3$. Since the latter two vertices are midpoints of $E_2$ and $E_3$, respectively, we have that $[p_i,p_j]$ is parallel to $[p_{2k},p_1]$, implying that $j=2k+1-i$. Let $q$ denote the midpoint of $E_1$, and recall that $q \in [p_{2k},p_1]$. Note that $\conv \{ q, p_i, p_j \}$ is the midtriangle of $Q$, and thus, the area of $Q$ does not depend on the position of $q$. This yields that as $\conv \{ q, p_i, p_j \}$ is the midtriangle of a triangle circumscribed about $P_k$, then, moving $q$ along the side $[p_1,p_{2k}$, we obtain that $\conv \{ p_1, p_i, p_j \}$ or $\conv \{ p_{2k}, p_i, p_j \}$ is also the midtriangle of a triangle circumscribed about $P_k$, and the area of this triangle is equal to $\area(Q)$. Thus, we may apply the argument from Subcase 2.1 and obtain the same quantities for $\area(Q)$ as in (\ref{eq:guess}).

\noindent
\emph{Subcase 2.3}, exactly two sides of $Q$ contain a side of $P_k$. We apply the approach of the previous cases, and let $\conv \{ p,q,r\}$ be the midtriangle of $Q$, where $p$ is a vertex of $P_k$. Then, like in Subcase 2.2, we can observe that both $[p,q]$ and $[p,r]$ are parallel to some sides of $P_k$. But this yields that both $q,r$ are vertices of $P_k$, and we obtain the estimate in (\ref{eq:guess}).

\noindent
\emph{Subcase 2.4}, every side of $Q$ contains a side of $P_k$. Let $T=\conv \{ p,q,r \}$ denote the midtriangle of $Q$. Since every side of $T$ is parallel to a side of $P_k$, the axial symmetry of $P_k$ with respect to the bisector of its every side and a simple geometric observation shows that $p,q,r$ are the midpoints of some sides of $P_k$. But then $T$ is axially symmetric with respect to the bisector of its every side, implying that $T$ is regular. From this we immediately obtain that $2k$ is divisible by $3$, and that $\area(Q)= 3 \tan \frac{\pi}{3}$.

We need to compare this quantity to the one in the case $r=0$ of (\ref{eq:guess}). 
Since
\[
3 \tan \frac{\pi}{3} = 6 \sin \left( \frac{2\pi}{3} \right) < \frac{ 6\sin  \left( \frac{2\pi}{3} \right) }{ \cos^2 \left( \frac{\pi}{2k} \right)}
\]
for all $k \geq 2$, the second part of Lemma~\ref{lem:regmain} immediately follows.
\end{proof}

Now we prove Theorem~\ref{thm:regularDowker}.

\begin{proof}[Proof of Theorem~\ref{thm:regularDowker}]
We need to determine which values of $k \geq 2$ satisfy the inequality
\begin{equation}\label{eq:regularforproof}
\frac{n-6}{n-m} \sqrt{A(k,m)} + \frac{6-m}{n-m} \sqrt{A(k,n)} \geq \sqrt{A(k,6)}
\end{equation}
for all $3 \leq m < 6 < n$.

Recall that by definition, $\BB^2 \subset P_{k}$ for all values of $k$, implying that $A_n(\BB^2) \leq A_n(P_k)$ for all values of $n$. An elementary computation shows that $A_n(\BB^2) = n \tan \frac{\pi}{n}$ for all $n$. Thus, for any $3 \leq m < 6 < n$, we have
\begin{multline*}
\frac{n-6}{n-m} \sqrt{A(k,m)} + \frac{6-m}{n-m} \sqrt{A(k,n)} \geq \frac{n-6}{n-m} \sqrt{m \tan \frac{\pi}{m}} + \frac{6-m}{n-m} \sqrt{n \tan \frac{\pi}{n}} \geq \\
\geq \frac{1}{2} \sqrt{5 \tan \frac{\pi}{5}} + \frac{1}{2} \sqrt{7 \tan \frac{\pi}{7}}.
\end{multline*}
Furthermore, for any $2k=6q+r$, where $q \geq 1$ and $0 \leq r < 6$,
\[
A(k,6)= (6-r) \tan \frac{\pi q}{2k} + r \tan \frac{\pi (q+1)}{2k}.
\]
Since $(6-r) \frac{\pi q}{2k} + r \frac{\pi (q+1)}{2k} = \pi$ and by the convexity of the function $x \mapsto \tan x$ on $(0,\frac{\pi}{2})$, we have that
for any fixed value of $r$, $A(k,6)$ is a decreasing sequence of $q$. Using this, an elementary computation shows that
\[
\sqrt{A(k,6)} \leq \frac{1}{2} \sqrt{5 \tan \frac{\pi}{5}} + \frac{1}{2} \sqrt{7 \tan \frac{\pi}{7}}
\]
for any $q \geq 3$ and $0 \leq r < 6$. Thus, (\ref{eq:regularforproof}) holds for any $k \geq 9$. We also observe that since $A(k,n) = \area(P_{k})$ for any $n \geq 2k$, (\ref{eq:regularforproof}) clearly holds if $k=2$ or $k=3$. Thus, we are left with the cases $k=4,5,6,7,8$, for which a direct computation of the two sides of (\ref{eq:regularforproof}) for any $3 \leq m < 6 < n \leq 2k$, using a Maple 18.00 software, yields the assertion.
\end{proof}

\section{Results about not necessarily polygonal norms}\label{sec:general}

In this section we collect our results about the $\alpha$-honeycomb properties of general norms, and the $\alpha$-Dowker properties of general convex disks.
For our first result, recall that $A_{\BB^2}(n) = n \tan \frac{\pi}{n}$. In the following theorem, we let
\[
\varepsilon_0 = \frac{\sqrt{A_{\BB^2}(5)}+ \sqrt{A_{\BB^2}(7)}-2\sqrt{A_{\BB^2}(6)}}{\sqrt{A_{\BB^2}(5)}+ \sqrt{A_{\BB^2}(7)}+2\sqrt{A_{\BB^2}(6)}} = 0.002623 \ldots ,
\]
and denote the Hausdorff distance of the convex bodies $K,L$ by $d_H(K,L)$.

\begin{theorem}\label{thm:stability}
Let $\M$ be a normed plane with unit disk $M$, and assume that $d_H(M, \BB^2) \leq \varepsilon_0$. Then $\M$ satisfies the honeycomb property.
\end{theorem}

\begin{proof}
Observe that $d_H(M,\BB^2) \leq \varepsilon$ implies that $(1-\varepsilon) \BB^2 \subseteq M \subseteq (1+\varepsilon) \BB^2$. On the other hand, since the isoperimetrix $M_{iso}$ of $\M$ is a rotated copy of the polar of $M$, this yields that $\frac{1}{1+\varepsilon} \BB^2 \subseteq M_{iso} \subseteq \frac{1}{1-\varepsilon} \BB^2$.

By Theorem~\ref{thm:generalized}, it is sufficient to prove that $M_{iso}$ satisfies the weak $\frac{1}{2}$-Dowker property. Since for any $n \geq 3$ and $K \subseteq L$, we have $A_K(n) \leq A_L(n)$, it follows that
$\frac{n \tan \frac{\pi}{n} }{(1+\varepsilon)^2} \leq A_{M_{iso}}(n) \leq \frac{n \tan \frac{\pi}{n} }{(1-\varepsilon)^2}$. This yields that it is sufficient to prove that for any $3 \leq m < 6 < n$ and $\varepsilon \leq \varepsilon_0$, we have
\[
\frac{1}{1+\varepsilon} \left( \frac{n-6}{n-m} \sqrt{m \tan \frac{\pi}{m}} + \frac{6-m}{n-m} \sqrt{n \tan \frac{\pi}{n}} \right) \geq \frac{1}{1-\varepsilon} \sqrt{6 \tan \frac{\pi}{6}}.
\]
By the properties of $A_{\BB^2}(n) = n \tan \frac{\pi}{n}$, the left-hand side is minimal if $m=5$ and $n=7$. Thus, it is sufficient to show that
\[
\frac{1}{1+\varepsilon} \left( \sqrt{5 \tan \frac{\pi}{5}} + \sqrt{7 \tan \frac{\pi}{7}} \right) \geq \frac{2}{1-\varepsilon} \sqrt{6 \tan \frac{\pi}{6}}
\]
holds for any $\varepsilon \leq \varepsilon_0$. But this readily follows from the definition of $\varepsilon_0$.
\end{proof}

Our next result shows that convex disks with sufficiently smooth boundaries `asymptotically' satisfy the log-Dowker property.

\begin{theorem}\label{thm:asymptotic}
If $K$ is a convex disk in $\Re^2$ with $C^4$-class boundary and strictly positive curvature everywhere, then there is some value $n(K) \in \Re$ such that for any $n \geq n(K)$, we have
\[
\log A_K(n-1) + \log A_K(n+1) \geq 2 \log A_K(n).
\]
\end{theorem}

In the proof we intend to use the following result of Ludwig \cite{Ludwig}.

\begin{lemma}[Ludwig]\label{thm:Ludwig}
Let $K$ be a convex disk in $\Re^2$ with $C^4$-class boundary and strictly positive curvature everywhere. Then there are quantities $A,B > 0$ and $C \in \Re$, depending only on $K$, such that
\[
A_K(n) = A + \frac{B}{n^2}+\frac{C}{n^4} + o \left( \frac{1}{n^4} \right).
\] 
\end{lemma}

\begin{proof}[Proof of Theorem~\ref{thm:asymptotic}]
Let us define $f(n) = A_K(n-1) A_K(n+1)-A^2_K(n)$. Then
\begin{multline*}
f(n) = \left( A + \frac{B}{(n-1)^2}+\frac{C}{(n-1)^4} + o \left( \frac{1}{n^4} \right) \right) \cdot \\
\cdot \left( A + \frac{B}{(n+1)^2}+\frac{C}{(n+1)^4} + o \left( \frac{1}{n^4} \right) \right) - \left( A + \frac{B}{n^2}+\frac{C}{n^4} + o \left( \frac{1}{n^4} \right) \right)^2 .
\end{multline*}
Thus, an elementary computation yields that
\begin{multline*}
f(n) = AB \left( \frac{1}{(n-1)^2} + \frac{1}{(n+1)^2} - \frac{2}{n^2} \right) + AC \left( \frac{1}{(n-1)^4} + \frac{1}{(n+1)^4} - \frac{2}{n^4} \right) + \\
  + B^2 \left( \frac{1}{(n-1)^2(n+1)^2} - \frac{1}{n^4} \right) + o  \left( \frac{1}{n^4} \right).
\end{multline*}
On the other hand, it can be checked that the orders of magnitude of $\frac{1}{(n-1)^4} + \frac{1}{(n+1)^4} - \frac{2}{n^4}$ and $\frac{1}{(n-1)^2(n+1)^2} - \frac{1}{n^4}$ are both $o  \left( \frac{1}{n^4} \right)$. Thus, after simplification, we obtain that
\[
f(n) = \frac{AB(6n^2-2)}{(n^2-1)^2n^2} + o  \left( \frac{1}{n^4} \right),
\]
which is positive if $n$ is sufficiently large.
\end{proof}




In Section~\ref{sec:polygonal} we have seen that there are polygons which do not satisfy the $\alpha$-Dowker property for any $\alpha < 1$ (see Remark~\ref{rem:onlyweak}). Our next theorem can be regarded as a counterpoint of this observation. Before stating it, recall that a convex disk $K$ is smooth if its every boundary point belongs to a unique supporting line of $K$.

\begin{theorem}\label{thm:betterthanone}
Let $K$ be smooth and strictly convex. Then there is some value $\alpha < 1$ such that $K$ satisfies the weak $\alpha$-Dowker property.
\end{theorem}

\begin{proof}
First, observe that by our conditions, any supporting line of $K$ intersects $K$ in a unique point, and for any point $p$ of $\bd(K)$ there is a unique supporting line of $K$ that contains $p$. This yields, in particular, that for any $n \geq 3$ and any minimum area convex $n$-gon $Q_n$ circumscribed about $K$, $Q_n$ has exactly $n$ sides, and the tangent points on those sides are distinct. For any point $p \in \bd(K)$, we denote the unique closed half plane supporting $K$ at $p$ by $H(p)$, and the boundary of $H(p)$ by $L(p)$. We call boundary points of $K$ with opposite unit normal vectors \emph{opposite}, or \emph{antipodal} points. Note that if $p,q \in \bd(K)$ are not antipodal points of $K$, then there is a unique closed arc of $\bd(K)$ connecting them that does not contain antipodal points. We denote this arc by $\widehat{pq}$, and denote by $A(\widehat{pq})$ the area of the bounded, connected region $R(\widehat{pq})$ whose boundary consists of $\widehat{pq}$ and the segments of $L(p)$ and $L(q)$ connecting $p,q$ to the intersection point of the two lines. We observe that the problem of finding $A_K(n)$ coincides with the problem of finding a tiling of $\bd(K)$ by $n$ arcs $\widehat{x_ix_{i+1}}$, $i=1,2,n$, $x_{n+1}=x_1$ such that $f=\sum_{i=1}^n A(\widehat{x_i x_{i+1}})$ is minimal.

First, we fix an arbitrary integer $n \geq 4$, and show that $A_K(n-1)+A_K(n+1) > 2 A_K(n)$. To do it, we follow the proof of the original theorem of Dowker \cite[Theorem 1]{CHD}.

Let $Q_{n-1}$ and $Q_{n+1}$ denote a minimum area convex $(n-1)$-gon and $(n+1)$-gon, respectively, circumscribed about $K$.
Let the tangent points of $K$ on the sides of $Q_{n-1}$ be $p_1,p_2, \ldots,p_{n-1}$ in counterclockwise order, and we denote the tangent points of $Q_{n+1}$ in counterclockwise order on $\bd(K)$ by $q_1,q_2,\ldots,q_{n+1}$, where the indices are understood mod $(n-1)$ and $(n+1)$, respectively. Note that then $Q_{n-1}=\bigcap_{i=1}^{n-1} H(p_i)$ and $Q_{n+1}= \bigcap_{j=1}^{n+1} H(q_j)$.
Since the areas of these regions are finite, for any two consecutive tangent points $p_i$, $p_{i+1}$, no arc $\widehat{p_ip_{i+1}}$ contains antipodal points of $K$, and the same is true for the points $q_j$.

Due to the possible existence of coinciding points in the above two sequences, we unite these sequences as a single sequence $v_1, v_2, \ldots, v_{2n}$ in such a way that the points are in this counterclockwise order in $\bd(K)$, $v_1=p_1$.
In the proof we regard this sequence as a cyclic sequence, where the indices are determined mod $2n$, and, with a little abuse of notation, we say that $\widehat{v_iv_j}$ \emph{covers} $\widehat{v_kv_l}$ if $\widehat{v_kv_l} \subseteq \widehat{v_iv_j}$ and $i < k < l < j < i+2n$.

Observe that the family of arcs $\widehat{p_ip_{i+1}}$, $\widehat{q_jq_{j+1}}$ is a $2$-tiling $\T_0$ of $\bd(K)$; that is, every point of $\bd(K)$ belongs to at least two such arcs, and no point belongs to the interior of more than two such arcs.
Our main goal will be to modify the $2$-tiling $\T_0$ in such a way that the value of $f$ does not increase but the number of covering pairs strictly decreases.

Note that since $\T_0$ is the union of two tilings consisting of $(n-1)$ and $(n+1)$ arcs, respectively, $\T_0$ contains covering pairs. 
Assume that $\widehat{v_iv_j}$ covers $\widehat{v_kv_l}$. Then let $\T_1$ denote the $2$-tiling of $\bd(K)$ in which $\widehat{v_iv_j}$ and $\widehat{v_kv_l}$
are replaced by $\widehat{v_iv_l}$ and $\widehat{v_kv_j}$.
Our main observation, as in the proof of \cite[Theorem 1]{CHD} is that in this case
\begin{equation}\label{eq:Dowkerold}
A(\widehat{v_iv_j}) + A(\widehat{v_kv_l}) \geq A(\widehat{v_iv_l}) + A(\widehat{v_kv_j}),
\end{equation}
and if the four points are pairwise distinct, then here we have strict inequality.

According to our conditions, $\sum_{S \in \T_0} A(S) \geq \sum_{S \in \T_1} A(S)$. Furthermore, as $\T_0$ and $\T_1$ are $2$-tilings, and $\widehat{v_kv_l}$ is already covered twice by the arcs considered in the modification, the number of covering pairs in $\T_1$ is strictly less than in $\T_0$. Repeating this procedure we obtain a $2$-tiling $\T_m$ of $\bd(K)$ for which $\sum_{S \in \T_0} A(S) \geq \sum_{S \in \T_,} A(S)$ and which does not contain covering pairs. Then, $\T_m$ decomposes into two tilings $\{ \widehat{v_1 v_3}, \widehat{v_3v_5}, \ldots, \widehat{v_{2n-1}v_1} \}$ and $\{ \widehat{v_2 v_4}, \widehat{v_4v_6}, \ldots, \widehat{v_{2n}v_2} \}$, each of which contains exactly $n$ arcs.

Let $V_1$ be defined as the circumscribed convex $n$-gon touching $K$ at the vertices $v_i$ with odd indices, and $V_2$ be the circumscribed convex $n$-gon touching $K$ at the vertices $v_i$ with even indices. Then
\begin{equation}\label{eq:areas}
\area(Q_{n-1})+\area(Q_{n+1}) \geq \area(V_1) + \area(V_2).
\end{equation}
Since $V_1$ and $V_2$ are convex $n$-gons circumscribed about $K$, if in (\ref{eq:areas}) we have strict inequality, or if $\area(V_i) > A_K(n)$ for $i=1$ or $i=2$, we are done. Thus, in the following we assume for contradiction that $\area(Q_{n-1})+\area(Q_{n+1}) = \area(V_1) + \area(V_2)$, and $\area(V_1)=\area(V_2)=A_K(n)$.
Then, by the remark right after (\ref{eq:Dowkerold}) about the equality case in the inequality, there is no arc $\widehat{p_ip_{i+1}}$ that contains points $q_j,q_{j+1}$ in its interior, and vice versa. In particular, there are at least two vertices of $Q_{n-1}$ that coincide with some vertices of $Q_{n+1}$.

Without loss of generality, let $p_2=q_2$ corresponding to $v_3$ and $v_4$, respectively, in the united sequence of vertices. Then $\{ p_1,q_1 \} = \{ v_1,v_2 \}$ and $\{p_3,q_3 \} = \{ v_5,v_6\}$. For $i=1,3$, let $p_i'$ denote the intersection point of $L(p_i)$ and $L(p_2)$, and let $q_i'$ denote the intersection point of $L(q_i)$ and $L(q_2)=L(p_2)$. Note that since $Q_{n-1}$ is a minimum area convex $(n-1)$-gon circumscribed about $K$, the tangent point on every side is the midpoint of the side. Thus, $p_2$ is the midpoint of $[p_1',p_3']$, and we obtain similarly that $q_2=p_2$ is the midpoint of $[q_1',q_3']$.

\begin{figure}[ht]
\begin{center}
\includegraphics[width=0.75\textwidth]{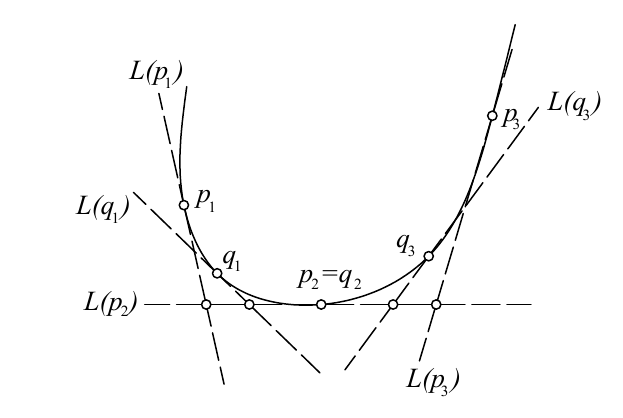}
\caption{The configuration in Case 1 of the proof of Theorem~\ref{thm:betterthanone}.}
\label{fig:dowker}
\end{center}
\end{figure}

\emph{Case 1}, $p_1=v_1$, $p_3=v_6$ (see Figure~\ref{fig:dowker}). Then, since $V_1, V_2$ are minimum area circumscribed $n$-gons, we have that $p_2=q_2$ is the midpoint of $[p_1',q_3']$ as well as that of $[q_1',p_3']$. But then $||p_3'-p_2||=||q_3'-p_2|| = ||p_1'-p_2|| = ||q_3'-p_2||$, implying that $p_1=q_1$ and $p_3=q_3$.

\emph{Case 2}, $p_1=v_2$, $p_3=v_5$. Then by a similar argument, we obtain that $p_1=q_1$ and $p_3=q_3$.

\emph{Case 3}, $p_1=v_1$ and $p_3=v_5$. Then we have that either $p_1=q_1$ and $p_3=q_3$, or $p_1 \neq q_1$ and $p_3 \neq q_3$. 

\emph{Case 4}, $p_1=v_2$ and $p_3=v_6$. Then we have that either $p_1=q_1$ and $p_3=q_3$, or $p_1 \neq q_1$ and $p_3 \neq q_3$. 

Now, let us call a point $q_j$ (resp. $p_i$) a \emph{double vertex}, if it coincides with some $p_i$ (resp. $q_j$).
If for some $j$, both $q_j$ and $q_{j+1}$ are double vertices, then $q_j$ is a double vertex for every $j$ by the previous consideration, which contradicts the fact that $Q_{n+1}$ has strictly more vertices than $Q_{n-1}$. If there is no value of $j$ such that both $q_j$ and $q_{j+1}$ are double vertices, then by the previous consideration any arc of $\bd(K)$ connecting two consecutive double vertices contains the same number of vertices from $Q_{n-1}$ and $Q_{n+1}$. But this also contradicts the fact that $Q_{n+1}$ has strictly more vertices than $Q_{n-1}$.

Until now, we have shown that for any $n \geq 4$, we have $A_K(n-1)+A_K(n+1)> 2 A_K(n)$. 
Thus, for any $n \geq 4$, either $A^{\bar{\alpha}}_K(n-1)+A^{\bar{\alpha}}_K(n+1)> 2 A^{\bar{\alpha}}_K(n)$ is satisfied for all $\bar{\alpha} > 0$, or there is a value $0 < \alpha_n < 1$ such that $A^{\bar{\alpha}}_K(n-1)+A^{\bar{\alpha}}_K(n+1) \geq 2 A^{\bar{\alpha}}_K(n)$ is satisfied if and only if $\bar{\alpha} \geq \alpha_n$.

Note that since $K$ is not a polygon, the sequence $\{ A_K(n) \}$ is strictly decreasing, and every element is strictly greater than $\area(K)$. Thus, there is some $\alpha_0 < 1$ and a positive integer $n(\alpha_0)>6$ such that for any $n > n(\alpha_0)$ and $m \in \{ 3,4,5\}$, the inequality (\ref{eq:weakalphadowker}) is satisfied for any $\alpha \geq \alpha_0$. Now, the quantity $\alpha = \max \{ \alpha_0, \alpha_1, \ldots, \alpha_{n(\alpha_0)} \} < 1$ satisfies the requirements in Theorem~\ref{thm:betterthanone}.
\end{proof}

\section{A conjecture of Steinhaus}\label{sec:Steinhaus}

We recall the following conjecture, appearing as Problem C15 in \cite{CFG}.

\begin{conjecture}[Steinhaus]
 For any tiling $\mathcal{T}$ in the Euclidean plane with tiles whose diameters are at least $D$ for some fixed $D > 0$, the maximum isoperimetric ratio $\frac{\perim(C)^2}{\area(C)}$ of the cells $C$ of $\mathcal{T}$  is minimal if $\mathcal{T}$ is a regular hexagonal tiling.
\end{conjecture}

To state our result regarding this conjecture, we need Definition~\ref{defn:isoperim}.

\begin{definition}\label{defn:isoperim}
Let $\mathcal{T}$ be a tiling of a normed plane $\M$.
Let $\mathcal{T}(R)$ denote the family of cells of $\mathcal{T}$ in $R \BB^2$. Then the \emph{lower average isoperimetric ratio} of a cell of $\mathcal{T}$ is defined as
\[
\underline{I}(\mathcal{T})= \liminf_{R \to \infty} \frac{\sum_{C \in \mathcal{T}(R)} \frac{\perim_M(C)^2}{\area(C)}}{ \card( \mathcal{T}(C)) }.
\]
If we replace the $\liminf$ in the above definition by $\limsup$, we obtain the \emph{upper average isoperimetric ratio} $\overline{I}(\mathcal{T})$ of a cell. If these quantities are equal, the common value is called the \emph{average isoperimetric ratio} of a cell, denoted by $I(\T)$.
\end{definition}

Our result is the following.

\begin{theorem}\label{thm:isoperimetric}
For any normed plane $\M$ there is a hexagonal tiling $\T_{hex}$ of $\M$ such that for any convex, normal tiling $\T$ of $\M$, we have
\[
\underline{I}(\T) \geq I(\T_{hex}).
\]
Furthermore, if $\M$ is a Euclidean plane, then $\T_{hex}$ is a regular hexagonal tiling.
\end{theorem}

\begin{proof}
 The first part of Theorem~\ref{thm:isoperimetric} follows from a straightforward modification of the proof of Theorem~\ref{thm:generalized}. For the second statement we remark that, according to this proof, the cells of $\T_{hex}$ are homothetic copies of a minimum area hexagon circumscribed about the isoperimetrix of the norm. Thus, if the norm is Euclidean, then these hexagons are regular.
\end{proof}

Theorem~\ref{thm:isoperimetric} can be stated also in the following, slightly more general form.

\begin{theorem}\label{thm:isoperimetric2}
For any normed plane $\M$ there is a hexagonal tiling $\T_{hex}$ of $\M$ such that for any convex tiling $\T$ of $\M$ satisfying $\overline{v}(\T) \leq 6$, we have
\[
\underline{I}(\T) \geq I(\T_{hex}).
\]
\end{theorem}

\begin{remark}
We note that Theorem~\ref{thm:isoperimetric2} remains valid if we replace $\overline{v}(\T)$ by $\underline{v}(\T)$ and $\underline{I}(\T)$ by $\overline{I}(\T)$. Furthermore, it is also valid if we replace the condition $\overline{v}(\T) \leq 6$ by the property that $\T$ has a cell with at most six sides, and the inequality $\underline{I}(\T) \geq I(\T_{hex})$ by $\sup \left\{ \frac{\perim_M(C)^2}{\area(C)} : C \hbox{ is a cell of } \T \right\} \geq I(\T_{hex})$.
\end{remark}

There is a construction in \cite{GS} showing that it is possible to tile the plane with convex heptagons. In this tiling, as we move away from the origin, the cells get longer and thinner, showing that it is not normal. On the other hand, an elementary geometric consideration proves Remark~\ref{rem:ratio}, where for any convex disk $K$ in $\Re^2$, we denote by $I(K)$, $\ir(K)$ and $\cirr(K)$ the isoperimetric ratio $\frac{\perim^2(K)}{\area(K)}$, the inradius and the circumradius of $K$ respectively; the latter two quantities are defined as the radii of a largest Euclidean disk contained in $K$ and of the smallest Euclidean disk containing $K$, respectively.

\begin{remark}\label{rem:ratio}
For any $\lambda > 0$ there is some $\mu > 0$ such that if $I(K) \leq \lambda$, then $\frac{\cirr(K)}{\ir(K)} \leq \mu$.
\end{remark}

\begin{proof}
Let $\F$ denote the family of convex disks, containing $o$, with unit perimeter and isoperimetric ratio at most $\lambda$. Then for any $K \in \F$, $\area(K) \geq \frac{1}{\lambda}$ and $K \subset \BB^2$, and thus, $\F$ is universally bounded. Since $\cirr(K) \leq 1$ for every $K \in \F$, we need to show that there is some $\tau > 0$ such that $\ir(K) \geq \tau$ for all $K \in \F$. Suppose for contradiction that this is not so. Then, there is a sequence $\{ K_n\}$ of elements of $\F$ such that $\ir(K_n) \to 0$. It is known (see e.g. p. 215, Ex. 6.2 of \cite{YB}) that for any convex disk we have $\ir(K) \geq \frac{w(K)}{3}$, where $w(K)$ is the minimal width of $K$. Thus, we have $w(K_n) \to 0$. But this, combined with $K_n \subset \BB^2$, clearly contradicts our assumption that $\area(K_n) \geq \frac{1}{\lambda}$ for all values of $n$.
\end{proof}

In light of Theorem~\ref{thm:isoperimetric2}, Remark~\ref{rem:ratio} and the construction in \cite{GS}, it seems meaningful to ask the following.

\begin{question}\label{ques:average}
Let $\T$ be a convex tiling of $\Re^2$ such that there are universal constants $\lambda, \mu > 0$ such that for every cell $C$ of $\T$, we have $\ir(C) \geq \lambda$ and $\frac{\cirr(C)}{\ir(C)} \leq \mu$. Is it true that $\overline{v}(\T) \leq 6$ (resp. $\underline{v}(\T)) \leq 6$? If the answer to both questions is negative, is it true that $\T$ contains a cell with at most six sides?
\end{question}

In the following example we show that if we use a square instead of $\BB^2$ in the definition of average, the answer for upper average is negative. To do it, let $\Sbf$ be the square with vertex set $\{ -1, 1 \}^2$, and for any convex tiling $\T$, with the notation $\T_{\square}(R) = \{ C \in \T : C \subseteq R \Sbf\}$, we set
\begin{equation}
\overline{v}_{\square}(\T) = \limsup_{R \to \infty} \frac{\sum_{C \in \T_{\square}(R)} v(C)}{\card( \T_{\square}(R))},
\end{equation}
where $v(C)$ denotes the number of sides of $C$.

\begin{example}\label{ex:loweraverage}
We define a convex tiling $\T$ of $\Re^2$ in consecutive steps. More specifically, we start with the family $\T_0 = \{ \Sbf \}$, and for every positive integer $k$, in the $k$th step we add some cells to the family $\F_{k-1}$ of the already defined cells to get the family $\F_k$. We will do it in such a way that after each step, $
\bigcup \F_k = N_k \Sbf$ for some positive integer $N_k$. Each time we use one of Step A and Step B, described below:

\noindent
\emph{Step A}: Let $\bigcup \F_{k-1}= N_{k-1} \Sbf$, and let the sides of this square be $E_1, E_2, E_3, E_4$. Consider a rectangle $R$ of side lengths $N_{k-1}$ and $2 N_{k-1}$. We set $\T_k = \T_{k-1} \cup \{ R_1, R_2, R_3, R_4 \}$, where each $R_i$ is congruent to $R$, and one half of a longer side of $R_i$ is $E_i$. We do it in such a way that the rectangles and $\bigcup \T_{k-1}$ are pairwise nonoverlapping (see Figure~\ref{fig:squares}).

\noindent
\emph{Step B}: Let $\bigcup \F_{k-1}= N_{k-1} \Sbf$. We fill the region obtained as the closure of $(N_{k-1}+2) \Sbf \setminus (N_{k-1} \Sbf)$ by $4N_{k-1}+4$ translates of $\Sbf$.

\begin{figure}[ht]
\begin{center}
\includegraphics[width=0.5\textwidth]{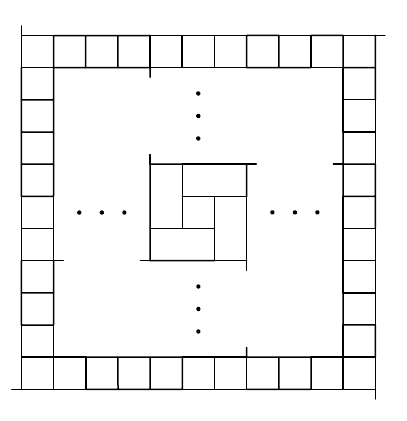}
\caption{An illustration for the construction in Example~\ref{ex:loweraverage}. The four rectangles surrounding the central square are obtained by applying Step A. The squares in the outermost layer are obtained by applying Step B.}
\label{fig:squares}
\end{center}
\end{figure}

In the following part we describe a suitable infinite sequence of $A$s and $B$s, where the k$th$ element is $A$ or $B$, if $\T_k$ is obtained from $\T_{k-1}$ by applying Step A or Step B, respectively. Then we set $\T = \bigcup_{k=1}^{\infty} \T_k$. Before describing this sequence, we remark that independently of its definition, if an edge $E_1$ of a cell $C_1 \in \T$ intersects an edge $E_2$ of another cell $C_2$ in a nondegenerate segment, then $E_1 \subseteq E_2$ or $E_2 \subseteq E_1$. Thus, we can apply an arbitrarily small deformation to the cells of $\T$ to obtain an edge-to-edge convex tiling $\T'$ of $\Re^2$ with the same combinatorial properties; this deformation can be done subsequently after each step. As a consequence, for simplicity, if a side of a polygon, appearing in one of the two steps above, is dissected into smaller segments by the vertices of the neighbors of the cell, we regard each smaller segment as a separate edge of the cell. 
Using this interpretation, in the following, we call the edges of $\T$ in $\bd \left( \bigcup \T_k \right)$ \emph{boundary edges} of $\T_k$, and the edges of $\T$ in $\inter \left( \bigcup \T_k \right)$ \emph{interior edges} of $\T_k$. 

Fix any positive number $\nu > 0$. First, consider the case that we apply Step A $k_1$ times, and after that we apply Step B once. Then $N_{k_1}=3^{k_1}$, and thus, $\T_{k_1}$ has $4\cdot 3^{k_1}$ boundary edges. In addition, $\T_{k_1}$ consists of $4k_1+1$ cells. We count the number of interior edges: Note that $\T_0$ has $4$ edges. Then in each step before the $k_1$th one we add $12$ edges, which are all interior edges of $\T_{k_1}$. Finally, in the $k_1$th step we add four interior edges of $\T_{k_1}$. Thus, in total, $\T_{k_1}$ has $12k_1-4$ interior edges. Now, if we set $R_1=N_{k_1}+\frac{1}{2}$, then
\[
v_{\square}(R_1) = \frac{\sum_{C \in \T_{\square}(R_1)} v(C)}{\card( \T_{\square}(R_1))} = \frac{4\cdot 3^{k_1}+2(12k_1-4)}{4k_1+1}.
\]
Observe that if $k_1$ is sufficiently large, then $v_{\square}(R_1) > \nu$. We choose $k_1$ to satisfy this inequality, and fix this value in the following.

Now we apply Step A $k_2-k_1-1$ times and after that we apply Step B once. Furthermore, we set $R_2 = N_{k_2}+\frac{1}{2}$. Note that if $k_2$ is sufficiently large, then the contribution of $\T_{k_1}$ to both the numerator and the denominator of the fraction defining $v_{\square}(R_2)$ is negligible, and thus, by the argument in the previus paragraph, we can choose some sufficiently large integer $k_2$ such that $v_{\square}(R_2) > \nu$.

Continuing this process, we obtain a tiling $\T$, and a sequence $\{ R_n \}$ with $R_n \to \infty$, such that $v_{\square}(R_n) > \nu$ for all values of $n$. This yields that $\overline{v}_{\square}(\T) \geq \mu$.
\end{example}

\section{An additional remark}\label{sec:remark}

Estimating the numbers of sides of cells of tilings has been a longstanding problem of geometry. The classical book of Gr\"unbaum and Shephard \cite{GS} shows that it is possible to tile the Euclidean plane with convex heptagons, but that any normal tiling contains infinitely many cells with at most six sides. Here we recall that a (not necessarily convex) tiling $\T$ of $\Re^2$ is called \emph{normal} if
\begin{itemize}
\item[(1)] every cell is homeomorphic to $\BB^2$;
\item[(2)] the intersection of every pair of cells is connected;
\item[(3)] the cells are \emph{universally bounded}, i.e. there are universal constants $0 < \hat{r} < \hat{R}$ such that every cell contains a Euclidean disk of radius $\hat{r}$ and is contained in a Euclidean disk of radius $\hat{R}$.
\end{itemize}
The statement of Gr\"unbaum and Shephard was reproved in a simpler way by Kazanci and Vince in \cite{KV}. 
For convex, normal mosaics, more is known. More specifically, the second part of Lemma~\ref{lem:Euler}, stating that the upper average number of sides of a cell of such a tiling is at most six is proved in the book \cite{Lagerungen} of L. Fejes T\'oth in a different way. The fact that if all cells have at least six sides, then the number of cells with more than six sides is finite is due to Stehling \cite{Stehling}. Akopyan \cite{Akopyan} gave a quantitative estimate for the number of these sides, while Frettl\"oh et al. \cite{FGL} showed that the order of magnitude given in \cite{Akopyan} is sharp.

We remark that the proofs of the Lemmas~\ref{lem:maxsides}-\ref{lem:Euler} (as well as Definition~\ref{defn:averagev}) can be generalized for all normal mosaics in a straightforward way. Thus, we gave a short proof of the following statement, generalizing the result of L. Fejes T\'oth for any normal mosaic.

\begin{theorem}\label{thm:upperaverage}
For any normal mosaic $\T$ in $\Re^2$, we have $\overline{v}(\T) \leq 6$.
\end{theorem}

\section{Acknowledgments}

The authors thank Frank Morgan for a valuable observation, stated as Remark~\ref{rem:Morgan} in the paper, and an unknown referee for many useful observations.

\end{document}